\documentclass[11pt]{amsproc}
\usepackage{amsmath, amssymb, amsthm, anysize, float}

\theoremstyle{plain}
\newtheorem{theorem}{Theorem}[section]
\newtheorem{lemma}[theorem]{Lemma}

\theoremstyle{remark}
\newtheorem{remark}[theorem]{Remark}
\newtheorem{construction}[theorem]{Construction}
\newenvironment{construction2}[1]{\vspace{10pt}\begin{construction}[#1]}{\end{construction}\vspace{10pt}}

\newcommand{\PG}{\mathsf{PG}}
\newcommand{\GF}{\mathbb{GF}}
\newcommand{\I}{\mathrm{I}}
\newcommand{\Four}{\mathcal{F}_4}
\newcommand{\Five}{\mathcal{F}_5}
\newcommand{\Six}{\mathcal{F}_6}
\newcommand{\pairs}[2]{\big\langle\begin{smallmatrix}#1\\#2\end{smallmatrix}\big\rangle}

\renewcommand{\O}{\mathrm{O}}

\begin{document}

\title{A geometric construction of Mathon's perp-system from four lines of $\mathsf{PG}(5,3)$}

\author{John Bamberg}\email[John Bamberg]{bamberg@maths.uwa.edu.au}
\address{Department of Pure Mathematics and Computer Algebra\\
Ghent University\\
Krijgslaan 281 -- S22\\
B-9000 Ghent\\
Belgium.}
\curraddr{School of Mathematics and Statistics\\
The University of Western Australia\\
35 Stirling Highway\\
Crawley, W. A., 6009\\
Australia.}

\author{Frank De Clerck}\email[Frank De Clerck]{fdc@cage.ugent.be}

\address{Department of Pure Mathematics and Computer Algebra\\
Ghent University\\
Krijgslaan 281 -- S22\\
B-9000 Ghent\\
Belgium.}

\subjclass[2000]{Primary 05B25, 51E12, 51E14}


\thanks{ This work was supported by a Marie Curie Incoming International Fellowship within
  the 6th European Community Framework Programme, contract number: MIIF1-CT-2006-040360.}

\begin{abstract}
We give a new construction of the Mathon perp-system of the five dimensional projective
space over the field with three elements, starting from four lines.
\end{abstract}

\maketitle

%
%

\section{Introduction}

There are certain maximal arcs of $\PG(2,q^2)$, $q$ even, constructed by Denniston
\cite{Denniston69} for which there is a non-degenerate orthogonal polarity of $\PG(2,q^2)$
such that for every point of the maximal arc, its polar line is an exterior line of the
arc. Such maximal arcs are called \textit{self-polar} and they give rise to interesting
partial geometries with parameters $\mathsf{pg}(q^2-1,(q^2+1)(q-1),q-1)$ (see
\cite{DeClerckUpdate}).  If we choose a basis for $\GF(q^2)$ over $\GF(q)$, we can
identify the vectors of $\GF(q^2)^3$ with the vectors of $\GF(q)^6$. Thus a maximal arc of
$\PG(2,q^2)$ corresponds to an \textit{SPG-regulus of lines} of $\PG(5,q)$, however, the
resulting polarity is degenerate\footnote{Note: The case
  $(P,Q)=(\mathsf{Q}(2n,q^2),\mathsf{Q}(4n,q))$, $q$ even, actually cannot arise in \cite[Theorem
    4]{DeClerckUpdate}, which is shown explicitly in a paper of Nick Gill \cite[Corollary
    2.5]{Gill}.}.  The notion of a \textit{perp-system} was introduced by \cite{DDHM} as a
particular type of \textit{self-polar} SPG-regulus of a finite projective space, and
moreover, one that produces a partial geometry under linear representation.  A perp-system
$\mathcal{P}$ of a projective space $\PG(d,q)$, with respect to a polarity $\rho$, is a
maximal set of mutually disjoint $r$-subspaces of $\PG(d,q)$ such that every two not
necessarily distinct elements of $\mathcal{P}$ are opposite (with respect to $\rho$). So
in particular, every element of a perp-system is non-singular with respect to $\rho$.  By
\textit{maximal} we mean that the cardinality of $\mathcal{P}$ attains the theoretical
upper bound of
$$q^{(d-2r-1)/2}\left(q^{(d+1)/2}+1\right)/\left(q^{(d-2r-1)/2}+1\right)$$ (see
\cite[Theorem 2.1]{DDHM}).  In \cite{DDHM}, a particular perp-system $\mathcal{M}$
attributed to Mathon was given as a set of $21$ lines of $\mathsf{PG}(5,3)$ with respect
to a symplectic polarity. It turns out that $\mathcal{M}$ is also a perp-system with
respect to some elliptic and some hyperbolic polarity, and its full stabiliser in
$\mathsf{PGL}(6,3)$ is $S_5$; a maximal subgroup of $\mathsf{PGSp}(6,3)$.  Mathon's
perp-system is the only known perp-system not arising from a maximal arc of $\PG(2,q)$.

In trying to understand the only known perp-system not arising from a maximal arc, the
authors found that everything begins with finding as set $\Four$ of four curious lines  of the projective space $\PG(5,3)$. Every
line of $\PG(5,3)$ can be represented by a $2\times 6$ matrix of rank $2$, and so we can
write these four lines so that they look like a frame of a projective plane (i.e., the
blow-up of a frame of $\PG(2,9)$):

\begin{table}[H]
\begin{center}
\begin{tabular}{c|cccc}
$\ell_1$ & $(\I$&$\O$&$\O)$\\
$\ell_2$ & $(\O$&$\I$&$\O)$\\
$\ell_3$ & $(\O$&$\O$&$\I)$\\
$\ell_4$ & $(\I$&$\I$&$\I)$
\end{tabular}
\end{center}
\caption{The set $\Four$ of four lines of $\PG(5,3)$. The matrices $\I$ and $\O$ denote
  the $2$-dimensional identity and zero matrices respectively over $\GF(3)$.}
\end{table}

In this note we give a geometric construction of Mathon's perp-system, starting from the
set of lines $\Four$ above. Via this construction, we also show that Mathon's perp-system
is related to the generalised quadrangle $\mathsf{W}(2)$. We should also note that
Mathon's perp-system provides insight into the embedding of $A_5$ as a maximal subgroup of
$\mathsf{PSp}(6,3)$. As far as the authors are aware, the only \textit{raison d'\^etre}
for this embedding comes from the character table of $\mathsf{SL}(2,5)$.  The faithful
absolutely irreducible ordinary character of $\mathsf{SL}(2,5)$ of degree $6$ has zero
$3$-defect and is rational, hence it is absolutely irreducible over $\GF(3)$. The
Frobenius-Schur indicator of this character is $-1$, hence both the ordinary
representation and its representation over $\GF(3)$ admit a symplectic form, which we can
see from the Brauer character table below (in mod $3$):

\begin{table}[H]\footnotesize
\begin{tabular}{cccccccccc}
& ind &1a& 2a& 4a&  5a& 10a&  5b& 10b\\
\hline
$\varphi_1$  & + &   $1$& $1$ & $1$ &  $1$ &  $1$ &  $1$&   $1$\\
$\varphi_2$  & + & $3$ & $3$ &$-1$&   $A$&   $A$&  $A^*$&  $A^*$\\
$\varphi_3$  & + &   $3$&  $3$ &$-1$&  $A^*$&  $A^*$&   $A$&   $A$\\
$\varphi_4$  & +  &  $4$ & $4$  &0&  $-1$&  $-1$&  $-1$&  $-1$\\
$\varphi_5$  & - &   $2$& $-2$&  $0$&  $-A$ &  $A$& $-A^*$&  $A^*$\\
$\varphi_6$  & - &  $2$& $-2$ & $0$& $-A^*$ & $A^*$&  $-A$ &  $A$\\
$\varphi_7$  & - &  $6$& $-6$ &$0$ &   $1$&  $-1$&   $1$ & $-1$\\
\hline
\end{tabular}
\caption{The Brauer character table for $\mathsf{SL}(2,5)$ modulo $\GF(3)$ (see
  \cite{ModAtlas}).  Note: $A = -\omega-\omega^4$, where $\omega$ is a primitive fifth
  root of unity, and $A^*$ is the Galois conjugate of $A$.}
\end{table}

%
%

\section{A Construction of Mathon's Perp-System}

A \textit{partial geometry} $\mathsf{pg}(s,t,\alpha)$ is a partial linear space such that
every line contains $s+1$ points, every point is incident with $t+1$ lines and for a point
$P$ and line $\ell$ which are not incident, there are $\alpha$ points on $\ell$ collinear
with $P$. We have already spoken fleetingly of partial geometries in the introduction
where we mentioned that maximal arcs and perp-systems give rise to partial geometries. In
turn, a partial geometry gives rise to a strongly regular graph
$$\mathfrak{srg}((s+1)(st+\alpha)/\alpha,s(t+1),s-1+t(\alpha-1),\alpha(t+1)),$$ and up to
duality, we only know of about eight constructions of partial geometries (see
\cite[4.1]{DeClerckUpdate}).  For $\alpha=2$, the only known partial geometries are
sporadic: the Van Lint-Schrijver, Haemers and Mathon partial geometries, the latter
arising from Mathon's perp-system $\mathcal{M}$.

From \cite{DDHM}, it is known that $\mathcal{M}$ consists of $21$ lines of
$\mathsf{PG}(5,3)$ forming a partial line-spread. The complement of $\mathcal{M}$ contains
precisely $21$ solids mutually intersecting in lines, and there are exactly three solids
through any point in this complement.  They showed by an exhaustive computer search, that
up to projectivity, there was only one such configuration of solids. This was enough for
Cristina Tonesi \cite{TonesiThesis} to ascertain that if a partial geometry with
$\alpha=2$ is a linear representation of a perp-system, then it is equivalent to the
partial geometry arising from the Mathon perp-system.  We will now develop a construction
of Mathon's perp-system from the four lines $\Four$. Let $X$ be the following
skew-symmetric matrix over $\GF(3)$:
$$\begin{pmatrix}
0&1\\
2&0
\end{pmatrix}.$$

\begin{lemma}
Let $\mathsf{W}(M)$ be a symplectic space $\mathsf{W}(5,3)$ defined by a Gram matrix
$M$. Then $\Four$ is a partial perp-system with respect to $\mathsf{W}(M)$ if and only if
$M$ is of the form
$$\begin{pmatrix}
\pm X&A&B\\
-A^T&\pm X&C\\
-B^T&-C^T&\pm X
\end{pmatrix}$$
for invertible matrices $A,B,C$ satisfying:
\begin{enumerate}
\item[(i)] the sum of the blocks of $M$ is nonzero;
\item[(ii)] $A+B-X$, $B+C-X$ and $C+A-X$ are invertible.
\end{enumerate}
\end{lemma}

\begin{proof}
That $\ell_1$, $\ell_2$ and $\ell_3$ are self-opposite implies that the three diagonal
blocks of $M$ are each equal to $\pm X$. Now consider $\ell_4$. Note that $\left( \I\,
\I\, \I\right)M\left( \I\, \I\, \I\right)^T$ is the sum of all the blocks of $M$, and we
require this to be nonzero. So we know that $M$ is of the form
$$\left(\begin{smallmatrix}
\pm X&A&B\\
-A^T&\pm X&C\\
-B^T&-C^T&\pm X
\end{smallmatrix}\right)$$
for some matrices $A,B,C$. For $\ell_1$, $\ell_2$ and $\ell_3$ to be opposite to one
another, we need that $A$, $B$, $C$ are each invertible. Condition (ii) follows from
requiring that $\ell_1$, $\ell_2$ and $\ell_3$ are opposite to $\ell_4$.
\end{proof}

%
%

\subsection{From Four to Five}

Let $\mathsf{W}(M)$ be the symplectic space defined by the Gram matrix:
$$\begin{pmatrix}
\O&X&X\\
X&\O&X\\
X&X&\O
\end{pmatrix}$$
and let $\perp$ denote the corresponding polarity.  If $\mathcal{K}$ is a set of subspaces
of a projective space, we will use the notation $\pairs{\mathcal{K}}{2}$ to denote the
spans of all distinct pairs of elements of $\mathcal{K}$.

\begin{lemma}\label{lemma:four}
Let $\mathcal{L}$ be the set of lines of $\mathsf{W}(M)$ which are disjoint from each
element of $\pairs{\Four}{2}$, and which for each $\ell\in\Four$ meets $\ell^\perp$ in a
point.  The set $\mathcal{L}$ has $24$ lines, and they can each be represented by a matrix
of the form
$$\begin{pmatrix}I&R&S\end{pmatrix}$$ where $R$, $S$, $R-I$ and $S-I$ are invertible, and
  $R+S$, $R+I$, $S+I$ have rank $1$.  Moreover, the element-wise stabiliser $H$ of $\Four$
  in $\mathsf{GL}(6,3)$ is the set of all diagonal matrices of the form
$$
\begin{pmatrix}
E&O&O\\
O&E&O\\
O&O&E
\end{pmatrix}
$$ where $E\in\mathsf{GL}(2,3)$. The quotient by the scalars yields a group isomorphic to
$S_4$, and this group acts transitively (under the induced projective action), and hence
regularly, on the set of $24$ elements of $\mathcal{L}$.
\end{lemma}

\begin{proof}
Consider an element $\ell$ of $\mathcal{L}$ and represent it as a $2\times 6$ matrix over
$\GF(3)$, split into three $2\times 2$ blocks:
$$\ell: \left(\begin{smallmatrix}
Q&R&S
\end{smallmatrix}\right)$$
Since $\ell$ is disjoint from the solid represented by $\left(\begin{smallmatrix}
  \I&\O&\O\\ \O&\I&\O\end{smallmatrix}\right)$, it follows that $S$ is
  invertible. Similarly, $Q$ and $R$ are invertible. Since $\ell$ is disjoint from the
  solid represented by $\left(\begin{smallmatrix}
    \I&\O&\O\\ \I&\I&\I\end{smallmatrix}\right)$, it follows that $S-R$ has rank $2$.
    Similarly, $R-Q$ and $Q-S$ are invertible.  Now $\ell_1^\perp$ has matrix
    $\left(\begin{smallmatrix} \I&\O&\O\\ \O&\I&-\I
\end{smallmatrix}\right)$
and so $R+S$ has rank $1$.  Similarly, the matrices $Q+R$ and $S+Q$ have rank $1$. Now we
look at $\ell_4^\perp$, which has matrix $\left(\begin{smallmatrix} \I&\O&-\I\\ \O&\I&-\I
\end{smallmatrix}\right),$
and we see that $Q+R+S$ has rank $1$.  Now $\ell$ must be totally isotropic, which is
equivalent to
$$RXQ+SXQ+QXR+SXR+QXS+RXS=0.$$ Without loss of generality, we may suppose that $Q$ is the
identity matrix.  Hence
\begin{enumerate}
\item[(i)] $R-I$ and $S-I$ are invertible;
\item[(ii)] $R+S$, $R+I$, $S+I$, and $R+S+I$ have rank $1$;
\item[(iii)] $RS+SX+XR+SXR+XS+RXS=0$.
\end{enumerate}
There are eight matrices $Y$ of $\mathsf{GL}(2,3)$ such that $Y-I$ is invertible and $Y+I$
has rank $1$, and they are
$$
\left(\begin{smallmatrix}0&1\\ 2&1\end{smallmatrix}\right),
\left(\begin{smallmatrix}0&2\\ 1&1\end{smallmatrix}\right),
\left(\begin{smallmatrix}1&1\\ 2&0\end{smallmatrix}\right),
\left(\begin{smallmatrix}1&2\\ 1&0\end{smallmatrix}\right),
\left(\begin{smallmatrix}2&0\\ 1&2\end{smallmatrix}\right),
\left(\begin{smallmatrix}2&0\\ 2&2\end{smallmatrix}\right),
\left(\begin{smallmatrix}2&1\\ 0&2\end{smallmatrix}\right),
\left(\begin{smallmatrix}2&2\\ 0&2\end{smallmatrix}\right).
$$ For each matrix $Y$ above, there are exactly three of the eight matrices which when
  added to $Y$ produce a matrix of rank $1$.  So we have $24$ pairs $(R,S)$ which satisfy
  the first two conditions above.  Moreover, one can check that the third condition is
  satisfied.
  
It is not difficult to see that the eight matrices underlying the set $\mathcal{L}$ in
fact form a conjugacy class of $\mathsf{GL}(2,3)$. Now suppose $E\in\mathsf{GL}(2,3)$ and
consider an element $(\I\, R\, S)$ of $\mathcal{L}$.  The image of $(\I\, R\, S)$ under
right multiplication by the element of $H$ arising from $E$ is $(E, RE,SE)$. Now observe
that if we were to normalise this $2\times 6$ matrix by elementary matrices so that we
generate the same line, we would need to apply the elementary matrices to the left of this
matrix. So $(E, RE,SE)$ represents the same line as $(\I, E^{-1}RE, E^{-1}SE)$. Therefore
$\mathcal{L}$ is $H$-invariant (under the induced projective action).  Clearly this action
is also transitive as the eight matrices underlying the set $\mathcal{L}$ form a conjugacy
class of $\mathsf{GL}(2,3)$.
\end{proof}

\begin{construction}
Let $\Five$ be the set $\Four$ augmented by one of the lines $(\I\, R\, S)$ of
$\mathcal{L}$.
\end{construction}

\begin{lemma}\label{lemma:five}
The setwise stabiliser of $\Five$ in $\mathsf{PGL}(6,3)$ is isomorphic to $S_5$, and this
group stabilises the symplectic forms defined by the Gram matrices
$$\begin{pmatrix}
\O&A&B\\
-A^T&\O&C\\
-B^T&-C^T&\O
\end{pmatrix}$$
where $A+B+C$ and $AR^T+BS^T+RCS^T$ are both symmetric matrices.
\end{lemma}

\begin{proof}
It is not difficult to show that the five lines determine the shape of the matrix above.
That is, any skew-symmetric matrix $M$ for which $\ell M\ell^T=0$ for all of our five
lines $\ell$, must be of the form stated.  For example, that $(\I\,\I\,\I)$ is totally
isotropic with respect to the symplectic form defined by a skew-symmetric matrix of the
form above, implies that $A+B+C$ is symmetric.

So now we just show that the stabiliser $G$ of $\Five$ is isomorphic to $S_5$.  Recall
from Lemma \ref{lemma:four} that the element-wise stabiliser $H$ of $\Four$ in
$\mathsf{GL}(6,3)$ is the set of all matrices of the form
$$
\begin{pmatrix}
E&O&O\\
O&E&O\\
O&O&E
\end{pmatrix}
$$ where $E\in\mathsf{GL}(2,3)$, and the quotient by the scalars yields a group isomorphic
to $S_4$ acting regularly on $\mathcal{L}$. So the action of $G$ on $\Five$ is faithful
and hence we know that $G$ is embedded in $S_5$. By the transitivity of $H$ on
$\mathcal{L}$, we may chose our favourite representative for $(I\,R\,S)$:
$$R:=\begin{pmatrix}2&0\\ 2&2\end{pmatrix},\quad 
S=\begin{pmatrix}2&2\\ 0&2\end{pmatrix}.$$

Now let
$$C=\begin{pmatrix}
O&O&C_0\\ 
O&C_0&O\\ 
C_0&O&O\\ 
\end{pmatrix},\quad
D=\begin{pmatrix}
1&1&0&0&0&0\\
1&0&0&0&0&0\\
0&0&0&0&2&2\\
0&0&0&0&0&1\\
0&1&0&1&0&1\\
1&0&1&0&1&0
\end{pmatrix}
$$ where $C_0=\left(\begin{smallmatrix}2&1\\ 0&1\end{smallmatrix}\right)$.  Then it can be
  easily checked that $C$ and $D$ satisfy the relations
$$C^2=1, \quad D^8=1,\quad [C,D^4]=1,\quad (CD)^5=1,\quad [C,D]^3=1$$ and hence $\langle
  C,D\rangle \cong \mathsf{SL}(2,5):2$ (see \cite{Atlas}). Moreover, $\langle C,D\rangle$
  stabilises $\Five$, and since this group must be maximal within the stabiliser of a
  symplectic form, $G$ is the projective image of $\langle C,D\rangle$.
\end{proof}

%
%

\subsection{From Five to Six}

Most of us are familiar with the isomorphism of $S_5$ with $\mathsf{PGL}(2,5)$ which gives
rise to two 5-transitive actions of $S_5$: on five and six points respectively.  To obtain
the action on six points directly from the action on five points, one simply notices that
$S_5$ has six subgroups of order $5$ which are permuted naturally under the conjugation
action of $S_5$; thus producing the multiply transitive action on $6$ elements. Likewise,
the stabiliser $S_5$ of $\Five$ determines a special set of six lines from
$\mathcal{L}$. The total number of lines of $\mathsf{PG}(6,3)$ is 11011, which you will
notice is congruent to $1$ modulo $5$.  Hence a $5$-cycle has a unique fixed line. The six
lines $\Six$ we want are precisely those lines of $\mathsf{PG}(5,3)$ which are fixed by
some $5$-cycle in $S_5$.


\begin{construction}
Let $\Six$ be those lines of $\mathsf{PG}(5,3)$ which are fixed by some $5$-cycle in
$S_5$.
\end{construction}

\begin{lemma}\label{lemma:six}
The lines $\Six$ form a partial perp-system of $\mathsf{W}(M)$ where $M$ is a matrix of
the form
$$\begin{pmatrix}
\O&A&B\\
-A^T&\O&C\\
-B^T&-C^T&\O
\end{pmatrix}$$
where $A+B+C$ and $AR^T+BS^T+RCS^T$ are both symmetric matrices.
\end{lemma}

\begin{proof}
By Lemma \ref{lemma:four}, we may chose our favourite representative for $(I\,R\,S)$,
since any other choice will result in a set of lines projectively equivalent to our choice
for $\mathcal{F}_5$. So let
$R:=\left(\begin{smallmatrix}2&0\\ 2&2\end{smallmatrix}\right)$ and
  $S:=\left(\begin{smallmatrix}2&2\\ 0&2\end{smallmatrix}\right)$.  Recall from Lemma
    \ref{lemma:five} that the stabiliser of $\mathcal{F}_5$ as a group of matrices is
    $\mathsf{SL}(2,5):2$ and is generated by $C$ and $D$ from the proof of Lemma
    \ref{lemma:five}.  The six five cycles we want in $\langle C,D\rangle$ are:
\begin{center}
\begin{tabular}{c|c|c}
$CD=\left(\begin{smallmatrix}
1&2&1&2&1&2\\
1&.&1&.&1&.\\
.&.&.&.&1&2\\
.&.&.&.&.&1\\
.&2&.&.&.&.\\
1&.&.&.&.&.
\end{smallmatrix}\right)$&
$DC=\left(\begin{smallmatrix}
.&.&.&.&2&2\\
&.&.&.&.&2&1\\
1&1&.&.&.&.\\
.&1&.&.&.&.\\
.&1&.&1&.&1\\
2&1&2&1&2&1
\end{smallmatrix}\right)$&
$D^{-1}CD^2=\left(\begin{smallmatrix}
1&2&.&1&2&.\\
1&.&2&.&2&2\\
.&.&.&.&1&1\\
.&.&.&.&.&1\\
.&2&.&2&.&2\\
1&.&1&.&1&.
\end{smallmatrix}\right)$\\&&\\
\hline&&\\
$D^2CD^{-1}=\left(\begin{smallmatrix}
.&.&2&2&.&.\\
.&.&1&.&.&.\\
.&.&.&.&1&1\\
.&.&.&.&1&2\\
1&2&1&2&1&2\\
1&.&1&.&1&.
\end{smallmatrix}\right)$&
$CD^{-1}CD^2C=
\left(\begin{smallmatrix}
2&2&2&2&2&2\\
2&1&2&1&2&1\\
1&2&.&.&.&.\\
.&1&.&.&.&.\\
.&2&1&1&.&1\\
1&1&1&2&2&1
\end{smallmatrix}\right)$&
$CD^2CD^{-1}C=
\left(\begin{smallmatrix}
.&1&.&1&.&1\\
2&1&2&1&2&1\\
.&1&.&.&.&.\\
2&.&.&.&.&.\\
.&.&1&.&.&.\\
.&.&2&1&.&.
\end{smallmatrix}\right)$
\end{tabular}
\end{center}
(Note: the periods above denote zero entries.)

We simply take the eigenspaces of these matrices (over $\GF(3)$) and find that the unique
fixed lines are accordingly
\begin{center}
\begin{tabular}{c|c|c}
$\begin{pmatrix}
1&.&1&2&2&.\\
.&1&1&.&2&2
\end{pmatrix}$
&
$\begin{pmatrix}
1&.&1&2&1&1\\
.&1&1&.&2&.
\end{pmatrix}$
& 
$\begin{pmatrix}
1&.&.&2&.&1\\
.&1&1&1&2&1
\end{pmatrix}$
\\ &&\\
\hline &&\\
$\begin{pmatrix}
1&.&.&2&2&.\\
.&1&1&1&2&2
\end{pmatrix}$
&
$\begin{pmatrix}
1&.&2&2&1&1\\
.&1&.&2&2&.
\end{pmatrix}$
&
$\begin{pmatrix}
1&.&2&2&.&1\\
.&1&.&2&2&1
\end{pmatrix}$.
\end{tabular}
\end{center}

If we let 
$$
A=\begin{pmatrix}
2&2\\
0&0
\end{pmatrix},
\quad
B=\begin{pmatrix}
0&0\\
1&1
\end{pmatrix},
\quad
C=\begin{pmatrix}
0&0\\
1&0
\end{pmatrix}
$$ be the block matrices for the Gram matrix $M$, then we see that these six lines are
non-degenerate with respect to the symplectic form defined by $M$, and moreover, they are
pairwise opposite.
\end{proof}

%
%

\subsection{From Six to Fifteen}

Recall that Mathon's perp-system contains 21 elements, which we can view as the sum of 6
and 15.  The relationship between 15 and 6 is simple; 15 is $6\choose 2$.

\begin{lemma}\label{lemma:fifteen}
Consider the symplectic space $\mathsf{W}(M)$ defined
by a Gram matrix $M$ of the form 
$$\begin{pmatrix}
\O&A&B\\
-A^T&\O&C\\
-B^T&-C^T&\O
\end{pmatrix}$$
where $A+B+C$ and $AR^T+BS^T+RCS^T$ are both symmetric matrices. Then for every $\alpha$
in $\pairs{\Six}{2}$, there is a unique line $\ell_\alpha$ contained in $\alpha$ with the
following properties:
\begin{enumerate}
\item[(i)] $\ell_\alpha$ is opposite to all elements of $\Six$;
\item[(ii)] $\ell_\alpha$ is disjoint from $\alpha'$
for all pairs $\alpha'\in\pairs{\Six}{2}$ different to $\alpha$.
\end{enumerate}
\end{lemma}

\begin{proof}
As in the proof of Lemma \ref{lemma:five}, we may use the transitivity of $H$ in Lemma
\ref{lemma:four} to concentrate on a particular choice of element $(I\,R\,S)$ of
$\mathcal{L}$ for out set $\Five$, and hence for our set $\Six$:
\begin{center}
\begin{tabular}{c|c|c}
$\left(\begin{smallmatrix}
1&.&1&2&2&.\\
.&1&1&.&2&2
\end{smallmatrix}\right)$
&
$\left(\begin{smallmatrix}
1&.&1&2&1&1\\
.&1&1&.&2&.
\end{smallmatrix}\right)$
& 
$\left(\begin{smallmatrix}
1&.&.&2&.&1\\
.&1&1&1&2&1
\end{smallmatrix}\right)$
\\ &&\\
\hline &&\\
$\left(\begin{smallmatrix}
1&.&.&2&2&.\\
.&1&1&1&2&2
\end{smallmatrix}\right)$
&
$\left(\begin{smallmatrix}
1&.&2&2&1&1\\
.&1&.&2&2&.
\end{smallmatrix}\right)$
&
$\left(\begin{smallmatrix}
1&.&2&2&.&1\\
.&1&.&2&2&1
\end{smallmatrix}\right)$.
\end{tabular}
\end{center}
It turns out, after some tedious calculation, that we get the following set of fifteen
lines, one $\ell_\alpha$ for each $\alpha\in\pairs{\Six}{2}$:
\begin{center}
\begin{tabular}{c|c|c|c|c}
$\left(\begin{smallmatrix}
1&.&2&2&2&2\\
.&1&.&2&.&2 
\end{smallmatrix}\right)$
&
$\left(\begin{smallmatrix} 
1&.&.&.&.&1\\
.&1&.&.&2&1  
\end{smallmatrix}\right)$
&
$\left(\begin{smallmatrix} 
1&.&1&.&2&2\\
.&1&.&1&.&2 
\end{smallmatrix}\right)$
&
$\left(\begin{smallmatrix} 
.&.&1&.&1&2 \\
.&.&.&1&1&. 
\end{smallmatrix}\right)$
&
$\left(\begin{smallmatrix} 
1&.&.&2&1&. \\
.&1&1&1&.&1 
\end{smallmatrix}\right)$
\\ &&&&\\
\hline &&&&\\
$\left(\begin{smallmatrix} 
.&.&1&.&1&1 \\
.&.&.&1&2&. 
\end{smallmatrix}\right)$
&
$\left(\begin{smallmatrix} 
1&.&2&.&.&. \\
.&1&2&2&.&. 
\end{smallmatrix}\right)$
&
$\left(\begin{smallmatrix} 
1&.&1&2&.&. \\
.&1&1&.&.&. 
\end{smallmatrix}\right)$
&
$\left(\begin{smallmatrix} 
1&.&1&.&.&. \\
.&1&.&1&.&. 
\end{smallmatrix}\right)$
&
$\left(\begin{smallmatrix} 
1&.&2&.&2&. \\
.&1&2&2&2&2 
\end{smallmatrix}\right)$
\\ &&&&\\
\hline &&&&\\
$\left(\begin{smallmatrix} 
1&.&.&.&1&. \\
.&1&.&.&.&1 
\end{smallmatrix}\right)$
&
$\left(\begin{smallmatrix} 
.&.&1&.&1&. \\
.&.&.&1&.&1 
\end{smallmatrix}\right)$
&
$\left(\begin{smallmatrix} 
1&.&.&.&2&2 \\
.&1&.&.&.&2 
\end{smallmatrix}\right)$
&
$\left(\begin{smallmatrix} 
1&.&2&.&1&. \\
.&1&2&2&.&1 
\end{smallmatrix}\right)$
&
$\left(\begin{smallmatrix}
1&.&1&.&1&1 \\
.&1&.&1&2&. 
\end{smallmatrix}\right)$.
\end{tabular}
\end{center}
The Gram matrix $M$ is, in this case,
$$\left(\begin{smallmatrix}
.&.&2&2&.&.\\
.&.&.&.&1&1\\
1&.&.&.&.&.\\
1&.&.&.&1&.\\
.&2&.&2&.&.\\
.&2&.&.&.&.
\end{smallmatrix}\right).$$
\end{proof}

Therefore, we arrive at the construction of Mathon's perp-system starting from our
original set of four lines $\Four$.

\begin{construction}[Mathon's perp-system]
Let $\mathcal{M}$ be $\Six$ together with the set of all fifteen lines $\ell_\alpha$
defined in Lemma \ref{lemma:fifteen}.
\end{construction}

%
%

\subsection{Conversely...}

We revised before that to derive the action of $S_5$ on six points given its action on
five points, we simply consider the action of $S_5$ on its $5$-cycles. To invert this
procedure, we take blocks of imprimitivity consisting of triples on the 15 pairs of elements from
$\{1,2,3,4,5,6\}$, and act naturally on these five sets of triples. (In fact, these
triples are known as \textit{synthemes}, which we will see in the following section).  So we
take triples of pairs of lines from $\Six$, and we get the set $\Five$ back again by
taking for each triple $T$, the intersection of the spans of the pairs in $T$.

%
%

\section{The generalised quadrangle $\mathsf{W}(2)$}

The smallest thick generalised quadrangle is $\mathsf{W}(2)$ and it has $15$ points and
$15$ lines. The most common construction of this generalised quadrangle is to take a
symplectic form on $V:=\GF(2)^4$, for example $\langle x,
y\rangle=x_1y_4+x_4y_1+x_2y_3+x_3y_2$, and then consider the two-dimensional subspaces of
$V$ for which this form vanishes. These two-spaces together with all of the nonzero
vectors of $V$ form the unique generalised quadrangle of order $2$ (up to isomorphism).
The oldest construction of $\mathsf{W}(2)$ is due to Sylvester (see \cite[Section
  6]{FGQ}).  Consider the six element set $\{1,2,3,4,5,6\}$. A \textit{duad} is an
unordered pair of distinct elements from this set and a \textit{syntheme} is a set of
three duads which are mutually disjoint. One obtains a generalised quadrangle of order $2$
by letting the duads be the points, synthemes as the lines and natural inclusion for
incidence. It can be easily shown that Sylvester's duad-syntheme geometry is a generalised
quadrangle of order $2$.

We will use Mathon's perp-system to construct $\mathsf{W}(2)$. Let $\pairs{\Six}{2}$ be
the set of $15$ solids spanned by two elements of $\Six$, and let $\perp$ denote the
symplectic polarity for which our $21$ lines are a perp-system.

\begin{construction2}{Generalised Quadrangle of Order $2$}\label{const:W2}\ \\
Let $\mathcal{F}_{15}$ be the fifteen lines induced by $\Six$ (see Lemma
\ref{lemma:fifteen}), and let $\mathcal{F}_{10}$ be the set of all lines of the form
$\langle \ell, m\rangle^\perp$ where $\ell,m$ are different elements of $\Five$.  Define
the following incidence structure of points and lines:
\begin{table}[H]
\begin{tabular}{lp{12cm}}
\textbf{Points:}& the elements of $\mathcal{F}_{15}$\\
\textbf{Lines (i):} & the elements of $\Five$\\
\textbf{Lines (ii):} & the elements of $\mathcal{F}_{10}$\\
\textbf{Incidence} & Let $P$ be a point. A line $\ell$ of type (i) is incident with
  $P$ if and only if $P\cap \ell^\perp=\varnothing$.  A line $\ell$ of type (ii)
  is incident with $P$ if and only if $P\subseteq \ell^\perp$.
\end{tabular}
\end{table}
\end{construction2}

\begin{theorem}
The incidence structure defined above is isomorphic to the generalised quadrangle
$\mathsf{W}(2)$.
\end{theorem}

\begin{proof}
We can verify this result by computer, but we outline below a sketch of a proof which does
not rely on the computer. To do this, we show that the geometry of Construction
\ref{const:W2} is isomorphic to Sylvester's syntheme-duad geometry. Recall that each line
of $\mathcal{F}_{15}$ is uniquely determined by a pair of elements of $\Six$.  If we label
the elements of $\Six$ as $\{m_1,m_2,\ldots,m_6\}$, then let $f$ be the
bijection $$m_{ij}\mapsto \{i,j\}$$ from $\mathcal{F}_{15}$ onto the set of duads, where
$m_{ij}$ is the unique element of $\mathcal{F}_{15}$ contained in $\langle
m_i,m_j\rangle$.  We also define now a bijection on the lines of the geometries.  Recall
that we can recover the action of $S_5$ on five elements by taking a block-system of
triples on the duads, and acting naturally on these five sets of triples. In other words,
we can find five synthemes which are mutually disjoint. Such a set of synthemes is called
a \textit{spread} which is the geometric term for a partition of the points into lines.
Now the permutation group induced by the stabiliser of $\Six$ admits a block system of
five synthemes on the indices of $\{m_1,m_2,m_3,m_4,m_5,m_6\}$, and hence we have a
certain set $\mathcal{S}$ of five synthemes. Label the five elements of $\mathcal{S}$ by
$s_1,s_2,s_3,s_4,s_5$.  Now for each pair of synthemes $s_a,s_b$ from $\mathcal{S}$, it is
not difficult to see that there is a unique syntheme $s_{ab}$ not in $\mathcal{S}$ which
is disjoint from $s_a$ and $s_b$.  So now we have a way to identify the two types of line
sets of Construction \ref{const:W2}.  Label the elements of $\mathcal{F}_5$ by
$\ell_1,\ell_2,\ell_3,\ell_4,\ell_5$.  Define a map $g$ from synthemes to the lines of
Construction \ref{const:W2} by assigning $g(s_a)$ to $\ell_a$ and by also assigning
$g(s_{ab})$ to $\ell_{ab}$ where $\ell_{ab}=\langle \ell_a, \ell_b\rangle^\perp$.

The next observation to make is that $$\ell_a=\bigcap_{\{u,v\}\subset s_a}\langle m_u,
m_v\rangle$$ for all $a\in\{1,2,3,4,5\}$.  We now proceed to show that incidence is
preserved through the maps $f$ and $g$.  Here is an outline of the two cases we need to
consider:
\begin{table}[H]
\begin{tabular}{l|l|l}
& Geometry of Construction \ref{const:W2} & Sylvester's Geometry\\
\hline
Points & $\mathcal{F}_{15}$, $m_{ij}$ & Duads $\{i,j\}$ \\
Lines (i) & $\Five$, $\ell_a$ & Spread of synthemes, $s_a$ \\
Lines (ii) & $\mathcal{F}_{10}$, $\ell_{ab}$ & Other synthemes, $s_{ab}$ \\
\hline
Incidence (I) & $m_{ij} \cap \ell_a^\perp=\varnothing$ & $\{i,j\} \subseteq s_a$\\
Incidence (II) & $m_{ij} \subseteq \langle \ell_a, \ell_b\rangle $ & $\{i,j\} \subseteq s_{ab}$\\
\hline
\end{tabular}
\end{table}

In the first case, note that for $a\in\{1,2,3,4,5\}$, we have 
$\ell_a^\perp=\left\langle m_u^\perp\cap m_v^\perp:\{u,v\}\subset s_a\right\rangle$.
Hence we see that the first row of incidence relations correspond as
$$m_{ij}\cap \left\langle m_u^\perp\cap m_v^\perp:\{u,v\}\subset
s_a\right\rangle=\varnothing \quad\text{if and only if}\quad \{i,j\}\subset s_a.$$
Similarly, the second row of incidence relations follow once one has observed that
$$m_{ij}\subseteq \langle\ell_a,\ell_b\rangle\quad\text{if and only if}\quad \{i,j\}\subset s_{ab}.$$
We leave the remaining details of this proof to the reader.
\end{proof}

\begin{remark}\
\begin{enumerate}
\item[(i)] Using the Sylvester's model of $\mathsf{W}(2)$, one can also construct (see
  \cite[Section 6]{FGQ}) the unique generalised quadrangle on 27 points and 45 lines.

\item[(ii)]  An affine embedding of $\mathsf{W}(2)$ into $\mathsf{AG}(3,3)$ is known \cite{Thas78}.
  We have not been able to establish a connection between this
  construction and ours, although they have the underlying field in common.
\end{enumerate}
\end{remark}

\section{Acknowledgements}

We would like to thank Rudi Mathon, Colva Roney-Dougal and J\"urgen M\"uller for various
discussions related to this work.

\providecommand{\bysame}{\leavevmode\hbox to3em{\hrulefill}\thinspace}
\providecommand{\MR}{\relax\ifhmode\unskip\space\fi MR }
\providecommand{\MRhref}[2]{%
  \href{http://www.ams.org/mathscinet-getitem?mr=#1}{#2}
}
\providecommand{\href}[2]{#2}

\bigskip

\end{document}